\documentclass[10pt,preprint]{imsart}

\usepackage{amssymb,amsmath,amsthm}
\usepackage[numbers]{natbib}
\usepackage{hyperref}
\usepackage{cleveref}
\usepackage{enumerate}

\usepackage[T1]{fontenc}
\usepackage{lmodern}
\usepackage{graphicx}
\usepackage{algorithm2e}
\usepackage{subfigure}
\usepackage{ifxetex,ifluatex}
\usepackage{fixltx2e} 

\usepackage{thmtools}
\usepackage{mathtools}

\usepackage{algorithmicx}
\hypersetup{breaklinks=true,
    linkcolor=blue,
    citecolor=magenta,
    colorlinks=true,
pdfborder={0 0 0}}

\usepackage{amsthm}
\usepackage{thmtools}
\usepackage{mathtools}

\usepackage{autonum}

\numberwithin{equation}{section}
\startlocaldefs

\crefname{thm}{theorem}{theorems}
\crefname{lemma}{lemma}{lemmas}
\crefname{prop}{proposition}{propositions}
\crefname{assumption}{assumption}{assumptions}
\crefname{example}{example}{examples}
\crefname{cor}{corollary}{corollaries}

\declaretheorem[name=Theorem,numberwithin=section]{thm}
\declaretheorem[name=Proposition,sibling=thm]{prop}
\declaretheorem[name=Lemma,sibling=thm]{lemma}
\declaretheorem[name=Corollary,sibling=thm]{cor}

\declaretheorem[name=Definition,numberwithin=section,style=definition]{definition}

\DeclareMathOperator*{\argmin}{argmin}

\DeclareMathOperator*{\diam}{diam}
\DeclareMathOperator*{\polylog}{polylog}

\newcommand{\E}{\mathbb{E}} 
\newcommand{\Evmu}{\mathbb{E}_\vmu} 
\newcommand{\R}{\mathbb{R}}
\newcommand{\Rn}{\mathbb{R}^n}

\newcommand{\vmu}{{\boldsymbol{\mu}}}
\newcommand{\vtheta}{{\boldsymbol{\theta}}}

\newcommand{\vc}{{\boldsymbol{1}}}

\newcommand{\vv}{{\boldsymbol{v}}}

\newcommand{\vzero}{{\boldsymbol{0}}}
\newcommand{\vx}{{\boldsymbol{x}}}

\newcommand{\vu}{{\boldsymbol{u}}}
\newcommand{\hmu}{\boldsymbol{\hat \mu}}

\newcommand{\vy}{\mathbf{y}}
\newcommand{\vxi}{{\boldsymbol{\xi}}}
\newcommand{\vg}{{\boldsymbol{g}}}

\newcommand{\increasings}{{\mathcal{S}_n^\uparrow}}

\newcommand{\decreasings}{{\mathcal{S}_n^\downarrow}}
\newcommand{\convexs}{\mathcal{S}_n^{\cup}}
\newcommand{\concaves}{\mathcal{S}_n^{\cap}}

\newcommand{\scalednorm}[1]{\Vert  #1 \Vert_n}

\newcommand{\scalednorms}[1]{\scalednorm{ #1 }^2}

\newcommand{\euclidnorm}[1]{\vert  #1 \vert_2}
\newcommand{\euclidnorms}[1]{\vert  #1 \vert_2^2}

\newcommand{\ls}{\hmu^{\textsc{ls}}}

\newcommand{\radiusInc}{{\hat r_\uparrow}}
\newcommand{\radiusUniform}{{\hat s_\uparrow}}
\newcommand{\radiusConv}{{\hat r_{\cup}}}
\newcommand{\kmax}{{k_{\textsc{max}}}}

\endlocaldefs

\begin{document}

\begin{abstract}
    A simple construction of adaptive confidence sets
    is proposed in isotonic, convex and unimodal regression.
    In univariate isotonic regression,
    the proposed confidence set enjoys uniform coverage over
    all non-decreasing regression functions.
    Furthermore, the diameter of the proposed confidence set
    automatically adapts to the unknown number of pieces of the true
    parameter, in the sense that the diameter is bounded
    from above by the minimax risk over the class of $k$-piecewise constant functions.
    The diameter of the confidence set is a simple increasing function
    of the number of jumps of the isotonic least-squares estimate.

    A similar construction is proposed in convex regression where
    the true regression function is convex and piecewise affine. Here,
    the confidence set enjoys uniform coverage and its diameter automatically
    adapt to the number of affine pieces of the true regression function.
    The diameter of the confidence set is an increasing function of
    the number of affine pieces of the convex least-squares estimate.

    We explain how to extend this technique to a non-convex set
    by proposing a similar adaptive confidence set in unimodal regression.
    The confidence set automatically adapts
    to the number of jumps of the true unimodal regression function
    and its diameter is an increasing function of the number of jumps
    of the unimodal least-squares estimate.

\end{abstract}

\title{Adaptive confidence sets in shape restricted regression}
\runauthor{P. C. Bellec}
\runtitle{Honest confidence sets in shape restricted regression}

\begin{aug}
    \author{\fnms{Pierre C.} \snm{Bellec}
        \thanksref{ecodec}
        \ead[label=e1,mark]{pierre.bellec@ensae.fr}
    }
    \affiliation{ENSAE and UMR CNRS 9194}
    \address{ENSAE,\\ 3 avenue Pierre Larousse, \\ 92245 Malakoff Cedex, France. }
    \thankstext{ecodec}{
        This work was
        supported by GENES and by the grant Investissements d'Avenir
    (ANR-11-IDEX-0003/Labex Ecodec/ANR-11-LABX-0047).}
\end{aug}

\date{\today}

\maketitle

\section{Introduction}
Let $K\subset\R^n$ be a closed convex set.
Assume that we have the observations
\begin{equation}
    Y_i = \mu_i + \xi_i, \qquad i=1,...,n,
\end{equation}
where the vector $\vmu =(\mu_1,...,\mu_n)^T \in K$ is unknown,
$\vxi = (\xi_1,...,\xi_n)^T$ is a noise vector with $n$-dimensional Gaussian distribution
$\mathcal{N}(\vzero,\sigma^2 I_{n\times n})$ where $\sigma>0$
and $I_{n\times n}$ is the $n\times n$ identity matrix.
Denote by $\Evmu$ and $\mathbb P_\vmu$ the expectation and the probability measure
corresponding
to the distribution of
the random variable $\vy=\vmu+\vxi$.
The vector $\vy = (Y_1,...,Y_n)^T$ is observed and the goal is to estimate $\vmu$.
Consider the scaled norm $\scalednorm{\cdot}$ defined by 
 \begin{equation}
     \scalednorms{\vu} = \frac{1}{n} \sum_{i=1}^n u_i^2, \qquad \vu = (u_1,...,u_n)^T\in\Rn.
\end{equation}
The error  of an estimator $\hmu$ of $\vmu$ 
is given by $\scalednorms{\hmu - \vmu}$. 
Let $\euclidnorms{\cdot}$ be the squared Euclidean norm, so that $\frac{1}{n} \euclidnorms{\cdot} = \scalednorms{\cdot}$.
For a finite set $E$,
let $|E|$ denote its cardinality.
We use bold face for vectors and the components of any vector $\vv\in\R^n$ are denoted by $v_1,...,v_n$. If $\vv\in\R^n$ and $T\subset\{1,...,n\}$, $\vv_T= (v_i,i\in T)\in\R^{|T|}$
is the restriction of $\vv$ to $T$.

In this paper, we consider the particular case where $K$ is a polyhedron,
that is, an intersection of a finite number of half-spaces.
If the true parameter $\vmu$ lies in a low-dimensional face of the polyhedron $K$,
it has been shown that for some polyhedra $K$,
the rate of estimation is of order $\frac{d\sigma^2}{n}$ up to logarithmic factors,
where $d$ is the dimension of the smallest face that contains $\vmu$
\cite{guntuboyina2013global,chatterjee2013risk,chatterjee2015matrix,bellec2015shape,han2017isotonic,guntuboyina2017spatial};
see also the survey \cite{guntuboyina2018nonparametric}.
This phenomenon appears, for example,
if the polyhedron $K$ is the cone of nondecreasing sequences \cite{chatterjee2013risk,bellec2015shape}
or the cone of convex sequences \cite{guntuboyina2013global,bellec2015shape}.
For these examples, if $\vmu$ lies in a $d$-dimensional face of the polyhedron $K$,
the Least Squares estimator over $K$ satisfies risk bounds and oracle inequalities
with the parametric rate $\frac{d\sigma^2}{n}$, up to logarithmic factors.
We consider the problem of confidence sets in this context.
In particular, the present paper addresses the following questions.
\begin{itemize}
    \item Is it possible to estimate or bound from below by a data-driven quantity
        the dimension $d$ of the smallest face of the polyhedron $K$ that contains the true parameter $\vmu$?
    \item Is it possible to construct a confidence set $\hat C_n$ such that:
        \begin{enumerate}
            \item It enjoys uniform coverage over all $\vmu\in K$ (i.e., $\vmu\in\hat C_n$ with high probability).
            \item It adapts to the smallest low-dimensional face that contains $\vmu$ (i.e., the diameter of $\hat C_n$ should be of the order $\frac{d\sigma^2}{n}$ 
                up to logarithmic factors
                if the smallest face that contains $\vmu$ has dimension $d$).
        \end{enumerate}
\end{itemize}
In this paper,
we answer these questions for two particular convex polyhedra,
the cone of nondecreasing sequences and the cone of convex sequences,
as well as for the non-convex set of unimodal sequences.

The construction of adaptive confidence sets in isotonic or convex regression
has been studied in \cite{dumbgen2003optimal,cai2006adaptive,cai2013adaptive}.
These papers show that if the true regression function is simultaneously smooth
and monotone, then it is possible to construct confidence sets that adapt to
the unknown smoothness of the true regression function.
In the present paper, there is no smoothness assumption
and the goal is to construct confidence sets that adapt
to the dimension $d$ of the smallest face of the polyhedron.

The rest of the paper is organized as follows.
\Cref{s:honest} gives the definition of honest and adaptive confidence sets.
\Cref{s:preli} defines the cone of nondecreasing sequences
and 
recalls some material from \cite{amelunxen2014living,mccoy2014steiner} on the statistical dimension and the intrinsic volumes of closed convex cones.
In \Cref{s:honest-inc} and \Cref{s:honest-conv}, we construct honest and adaptive
confidence sets for the cone of nondecreasing sequences
and for the cone of convex sequences, respectively.
In \Cref{s:unimodal}, we show that a similar construction is possible
for the non-convex set of unimodal sequences,
which shows that our construction of confidence sets
does not require convexity of the model.

\section{Honest and adaptive confidence sets}
\label{s:honest}

Let $(E_k)_{k\in J}$ be a collection of subsets $K$ indexed by some possibly infinite set $J$.
We will refer to the sets $(E_k)_{k\in J}$ as the \textit{models}.
If $J=\{1,...,\kmax \}$,
these models may be ordered by inclusion so that
\begin{equation}
    E_1\subset ... \subset E_{\kmax}= K.
    \label{eq:ordered-by-inclusion}
\end{equation}
For any model $E_k\subset K$, the minimax risk on $E_k$ is the quantity
\begin{equation}
    \label{eq:minimax-risk}
    R^*_{\E}(E_k) = \inf_{\hmu} \sup_{\vmu\in E_k} \Evmu \scalednorms{\hmu - \vmu},
\end{equation}
where the infimum is taken over all estimators, that is, all
random variables of the form $\hmu=g(\vy)$ where $g:\R^n\rightarrow\R^n$ is a Borel function.
If $J=\{1,...,\kmax \}$ and \eqref{eq:ordered-by-inclusion} holds,
the minimax risks satisfy
\begin{equation}
    R^*_\E (E_1) \le ... \le R^*_\E (E_\kmax).
\end{equation}
In that case, the collection $(E_k)_{k=1,...,\kmax}$ represents models of increasing complexity.

Similarly,
if a confidence value $\alpha\in(0,1)$ is given, one may
define the minimax quantity
\begin{equation}
    \label{eq:minimax-loss}
    R^*_\alpha(E_k) = \inf\left\{ R>0:
            \quad
            \sup_{\hmu} \inf_{\vmu\in E_k} \mathbb P_\vmu
            \left(
                \scalednorms{\hmu - \vmu}
                \le R
            \right) \ge 1 - \alpha
        \right\}
\end{equation}
for all $k\in J$,
where the supremum of taken over all estimators.
This quantity represents the smallest size, in a minimax sense, of a confidence ball with confidence level $1 - \alpha$.
Similarly, if $J=\{1,...,k_{\max}\}$
and the models are ordered by inclusion as in \eqref{eq:ordered-by-inclusion},
this quantity is an increasing function of $k$ and we have
\begin{equation}
    R^*_\alpha (E_1) \le ... \le R^*_\alpha (E_\kmax)
\end{equation}
for all $\alpha\in(0,1)$.

The goal of this paper is to study confidence sets in shape restricted regression.
A confidence set is a region $\hat C_n$ such that with high probability,
the unknown parameter $\vmu$ belongs to $\hat C_n$.
Let $\alpha\in(0,1)$.
If $\vmu\in E_{k^*}$ for some $k^*\in J$,
the quantity \eqref{eq:minimax-loss} may be used to define 
the oracle region
\begin{equation}
    \hat C_n^*(k^*) \coloneqq\{ \vu\in \R^n: \scalednorms{\vu-\hmu} \le R^*_\alpha(E_k) \},
\end{equation}
where $\hmu$ is an estimator that achieves the supremum in \eqref{eq:minimax-loss} 
(we assume here that all infima and suprema in \eqref{eq:minimax-loss} are attained).
Then, by definition of $R^*_\alpha(\cdot)$, we have
that $\vmu\in \hat C_n^*(k^*)$ with probability at least $1-\alpha$.
We call $\hat C_n^*(k^*)$ an \textit{oracle} region since it is inaccessible
for two reasons.
First, the radius $R^*_\alpha(E_{k^*})$ and the integer $k^*$ must be known
in order to construct $\hat C_n^*(k^*)$, i.e.,
the knowledge of the smallest model that contains $\vmu$ is needed.
Second, the oracle region $\hat C_n^*(k^*)$ is an Euclidean ball centered
around the estimator $\hmu$ that achieves the infimum in
\eqref{eq:minimax-loss}, and this estimator is unknown.

This paper studies the construction of data-driven confidence sets $\hat C_n$.
We consider only $1-\alpha$ confidence sets,
which means that the true parameter $\vmu$
belongs to $\hat C_n$ with probability at least $1-\alpha$,
uniformly over all $\vmu\in K$.

We also want the diameter of the confidence set $\hat C_n$ to be of the same order as
the diameter of the oracle region $\hat C_n^*(k^*)$, that is, the value $R^*_\alpha(E_{k^*})$.
Furthermore
the construction of $\hat C_n$ should not require the knowledge of
the smallest model that contains the true parameter $\vmu$:
The knowledge of $k^*$ is not needed to construct the confidence region $\hat C_n$.
In that case, we say that the confidence set $\hat C_n$ is adaptive.

We now give a formal definition of these properties.
For any $A\subset\R^n$, define the diameter of $A$ for the scaled norm
$\scalednorm{\cdot}$ by
\begin{equation}
    \diam A \coloneqq \sup_{\vv, \vu\in A} \scalednorm{\vv - \vu}.
\end{equation}

\begin{definition}
    \label{def:hci}
    Let $\alpha\in(0,1)$.
    Let $K\subset\R^n$ be a closed convex set and 
    let $(E_k)_{k\in J}$ be a collection of subsets of $K$ indexed by an arbitrary set $J$.
    Let $\hat C_n = \hat C_n(\vy)$ be a Borel subset of $\R^n$
    measurable with respect to $\vy$.
    We say that $\hat C_n$ is an honest confidence set if
    \begin{align}
        \inf_{\vmu\in K}
        \mathbb P_\vmu \left( \vmu \in \hat C_n \right)
        \ge 1 - \alpha.
        \label{goal:confidence}
    \end{align}
    We say that an honest confidence set $\hat C_n$ is adaptive in probability if for all $\gamma\in(0,1)$,
    \begin{align}
        \inf_{k\in J}
        \inf_{\vmu\in E_k}
    \mathbb P_\vmu \left( \diam(\hat C_n)^2 \le c' R^*_\alpha(E_k) \log\left(\frac{en}{\gamma\alpha}\right)^c \right)\ge 1 - \gamma,
        \label{goal:minimax}
    \end{align}
    where $c'>0$ and $c\ge0$ are numerical constants.
    Alternatively to \eqref{goal:minimax},
    we say that the confidence set $\hat C_n$ is adaptive in expectation if
    for all $k\in J$,
    \begin{align}
        \sup_{\vmu\in E_k}
        \Evmu \left[ \diam(\hat C_n)^2 \right] \le c' R^*_\E(E_k) \log\left(\frac{en}{\alpha} \right)^c,
        \label{goal:minimax-E}
    \end{align}
    where $c'>0$ and $c\ge0$ are numerical constants.
\end{definition}
The role of the constant $c\ge0$ is to allow for logarithmic factors.
The statistic $\hat C_n$ induces an honest confidence set:
If the definition above holds, \eqref{goal:confidence} 
says that the true sequence $\vmu$ lies in $\hat C_n$
with high probability.
Inequality \eqref{goal:minimax} implies that if the true
parameter satisfies $\hmu\in E_{k^*}$ for some $k^*\in J$,
then the diameter of $\hat C_n$
is of the same order as the minimax quantity \eqref{eq:minimax-loss}
of the model $E_k$,
up to logarithmic factors.

We now consider a special case: confidence balls centered at the Least Squares estimator.
The Least Squares estimator
over a closed convex set $K$ is defined by
\begin{equation}
    \ls(K) = \argmin_{\vu\in{K}} \scalednorms{\vy - \vu} = \Pi_K(\vy)
\end{equation}
where $\Pi_K$ denotes the convex projection onto $K$.
By definition of the convex projection onto $K$,
we have $(\vu-\Pi_K(\vy))^T(\vy-\Pi_K(\vy))\le 0$ for all $\vu\in K$,
which can be rewritten as
\begin{equation}
    \scalednorms{\ls(K) - \vy}
    \le
    \scalednorms{\vu - \vy}
    -
    \scalednorms{\vu - \ls(K)}.
    \label{eq:strong-convexity-confidence}
\end{equation}

If the confidence set $\hat C_n$ is an Euclidean ball,
it is characterized by its center and its radius.
Let $\alpha\in(0,1)$ be a confidence value, typically $\alpha=0.05$.
Let $K\subset\R^n$ be a closed convex set and 
let $(E_k)_{k\in J}$ be a collection of subsets of $K$ indexed by an arbitrary set $J$.
Let $\hat r$ be a positive random variable measurable with respect to $\vy$
and let $\ls(K)$ be the Least Squares estimator over $K$.
The set
\begin{equation}
    \hat C_n = \{\vv\in\R^n: \scalednorms{\ls(K)  - \vv} \le \hat r \}
    \label{eq:Cn-around-ls}
\end{equation}
is an honest confidence ball
if \eqref{goal:confidence} holds.
The confidence ball $\hat C_n$ is said to be adaptive in probability if 
\eqref{goal:minimax} holds, that is,
for all $\gamma\in(0,1)$,
\begin{align}
    \inf_{k\in J}
    \inf_{\vmu\in E_k}
    \mathbb P_\vmu \left( \hat r\le c' R^*_\alpha(E_k) \log\left(\frac{en}{\gamma\alpha}\right)^c \right)\ge 1 - \gamma,
    \label{goal:minimax-ball}
\end{align}
for all $\gamma\in(0,1)$ where $c'>0$ and $c\ge0$ are numerical constants.
The confidence ball $\hat C_n$ is said to be adaptive in expectation if
\eqref{goal:minimax-E}, that is,
\begin{align}
    \sup_{\vmu\in E_k}
    \Evmu [\hat r] \le c' R^*_\E(E_k) \log\left(\frac{en}{\alpha}\right)^c,
    \label{goal:minimax-ball-E}
\end{align}
for all $k\in J$,
where $c'>0$ and $c\ge0$ are numerical constants.

\section{Preliminaries}
\label{s:preli}

\subsection{The cone of nondecreasing sequences and the models $\left(\increasings(k)\right)_{k=1,...,n}$}
\label{s:inc}

Let $\increasings$ be the set of all nondecreasing sequences,
defined by
\begin{align}
    \increasings & \coloneqq \{\vu=(u_1,...,u_n)^T\in\Rn:  u_i \le u_{i+1}, \quad i=1,...,n-1\}.
\end{align}
For $n=1$, let $\increasings = \R$.
For all $n\ge1$, define the cone of non-increasing sequences by
$\decreasings \coloneqq - \increasings$.

For any $\vu\in\increasings\cup\decreasings$,
let 
$k(\vu)\coloneqq |\{u_i, i=1,...,n\}|$ where $|A|$ denotes the cardinality of set $A$.
The integer $k(\vu)$ is the smallest positive integer
such that $\vu$ is piecewise constant with $k(\vu)$ pieces.
The integer $k(\vu)-1$ is also
the number of jumps of $\vu$,
that is, the number of inequalities
$u_i \le u_{i+1}$ that are strict.
Define the sets
\begin{equation}
    \increasings(k) = \{ \vu\in\increasings: k(\vu) \le k \},
    \qquad
    k=1,...,n.
\end{equation}
The set $\increasings(1)$ is the subspace of all constant sequences
while $\increasings(2),...,\increasings(n-1)$ are closed non-convex sets.
We have
\begin{equation}
    \increasings(1)
    \subset
    \increasings(2)
    \subset
    ...
    \subset
    \increasings(n)
    =
    \increasings.
\end{equation}
The minimax risk over the sets $\increasings(k)$ satisfies
\begin{equation}
    c'' \sigma^2 k \log\log(16 n/k) / n
    \le
    R^*_{\E}(\increasings(k)) 
    \le
    c'''\sigma^2 k \log\log(16 n/k) / n,
    \label{eq:risk-increasings}
\end{equation}
for $k\ge 2$ and some absolute constant $c'',c'''>0$,
cf.  \cite{gao2017minimax}.
Furthermore, for any 
$\vmu\in\increasings(k)$, the Least-Squares estimator satisfies
$\Evmu \scalednorms{\ls(\increasings) - \vmu} 
    \le
    \sigma^2 k \log(en/k) / n$
by \cite[Theorem 2]{bellec2015shape}, hence
$\ls(\increasings)$ achieves the minimax rate, up to logarithmic factors.
Regarding results in probability, it is known 
that there exist numerical constants $c,c'$ such that for all $\alpha \le c$,
\begin{equation}
    c'\sigma^2k / n
    \le R^*_\alpha(\increasings(k)) 
    \le
    2\sigma^2 k \log(en/k) / n
    + 10\log(1/\alpha) / n,
    \label{eq:minimax-loss-increasings}
\end{equation}
cf. \cite[Proposition 4]{bellec2015sharp} for the lower bound
and \cite{bellec2015shape} for the upper bound.
Thus, for $\alpha>0$ small enough, the quantity
$R^*_\alpha(\increasings(k))$ is of order $k\sigma^2/n$,
up to logarithmic factors in $n$ and $1/\alpha$.

\subsection{Statistical dimension and intrinsic volumes of cones}

We recall here some properties of closed convex cones.
Most of the material of the present section
comes from \cite{amelunxen2014living,mccoy2014steiner}.
In the present paper, a cone is always pointed at 0.
A polyhedral cone is a closed convex cone of the form
\begin{equation}
    \label{eq:def-polyhedral}
    K = \{ \vu\in\R^n:  \vu^T\vv_j \le 0 \text{ for all } j=1,...,k \},
\end{equation}
where $\vv_1,...,\vv_k$ are vectors in $\R^n$,
that is, $K$ is the intersection of a finite number of half-spaces.
The dual or polar cone of $K$ is defined as
\begin{equation}
    K^\circ \coloneqq \{ \vtheta\in\R^n: \vv^T\vtheta \le 0 \text{ for all } \vv\in K\}.
\end{equation}
If $K$ a polyhedral cone, the face of $K$ with outward vector $\vtheta\in\R^n$
is the set
\begin{equation}
    F(\vtheta) \coloneqq \{ \vu\in K: \vu^T \vtheta = \sup_{\vv\in K} \vv^T\vtheta \}.
\end{equation}
The face $F(\vtheta)$ is nonempty if and only if $\vtheta\in K^\circ$.
If $K$ is the polyhedral cone \eqref{eq:def-polyhedral}
defined by the vectors $\vv_1,...,\vv_k$,
a face of a polyhedral cone $K$ has to be of the form
\begin{equation}
    \{ \vu\in K: \vu^T\vv_j = 0 \text{ for all } j\in T\}
    \label{eq:face-form}
\end{equation}
for some $T\subset\{1,...,k\}$.
The dimension of a face $F$ is the dimension of the affine span of $F$.

\begin{definition}[Statistical dimension, \citet{amelunxen2014living}]
    For any closed convex cone $K\subset\R^n$,
    define
    \begin{equation}
       \delta(K) 
       \coloneqq
       \E \left[\euclidnorms{\Pi_K(\vg)}\right]
       =
       \E\left[ \vg^T \Pi_K(\vg)\right]
       =
       \E \left[
               \left(
        \sup_{\vtheta\in K: \euclidnorm{\vtheta}\le 1}
            \vg^T \vtheta
        \right)^2
        \right]
       ,
    \end{equation}
    where $\vg\sim\mathcal N (\vzero,I_{n\times n})$. 
    The quantity $\delta(K)$ is called
    the statistical dimension of the cone $K$.
\end{definition}
It is also known that the following holds almost surely
\begin{equation}
\label{eq:almost-sure-delta}
   \euclidnorms{\Pi_K(\vg)}
   =
   \vg^T \Pi_K(\vg)
   =
           \left(
    \sup_{\vtheta\in K: \euclidnorm{\vtheta}\le 1}
        \vg^T \vtheta
    \right)^2
   ,
\end{equation}
cf. \cite[Proposition 3.1]{amelunxen2014living}.
The random variable \eqref{eq:almost-sure-delta}
concentrates around its expectation.
Indeed, the function $\vg\to \euclidnorm{\Pi_K(\vg)}$ is 1-Lipschitz
so that the Gaussian concentration theorem \cite[Theorem 5.5]{boucheron2013concentration} states that
$\euclidnorm{\Pi_K(\vg)} \le \E \euclidnorm{\Pi_K(\vg)} + \sqrt{2\log(1/\alpha)}$
holds with probability at least $1-\alpha$.
Since $\E \euclidnorm{\Pi_K(\vg)}\le \delta(K)$ by Jensen's inequality,
this implies a concentration result for the random variable \eqref{eq:almost-sure-delta}:
on the same event of probability at least $1-\alpha$, one has
\begin{equation}
   \euclidnorms{\Pi_K(\vg)}
   \le
   \delta(K)
   + 2 \sqrt{ 2 \log(1/\alpha) \delta(K) }
   + 2 \log(1/\alpha)
   \le
   2 \delta(K)
   + 4 \log(1/\alpha).
   \label{eq:concentration-squarednorm}
\end{equation}

We now define the intrinsic volumes
of a polyhedral cone, which are closely related
to the statistical dimension.
\begin{definition}[Intrinsic volumes of a polyhedral cone \cite{amelunxen2014living}]
    Let $K\subset\R^n$ be a polyhedral cone
    and let $\vg\sim\mathcal N(\vzero,I_{n\times n})$.
    The intrinsic volumes of $K$ are the real numbers
    \begin{equation}
        \nu_k(K) = \mathbb P \left(
            \Pi_K(\vg)
            \text{ lies in the relative interior of a $k$-dimensional face of }
        K \right),
    \end{equation}
    for all $k=0,...,n$.
\end{definition}
The intrinsic volumes of a polyhedral cone $K$ define
a probability distribution on the discrete set $\{0,...,n\}$.
More precisely, define the random variable
\begin{equation}
    \label{eq:def-V_K}
    V_K =  \sum_{k=0}^n k \; \mathbf{1}_{
        \{
            \Pi_K(\vg)
            \text{ lies in the relative interior of a $k$-dimensional face of }
            K
        \}
        },
\end{equation}
where $\mathbf 1_{\{\cdot\}}$ is the indicator function.
The random variable $V_K$ is valued in $\{0,...,n\}$ 
and satisfies $\mathbb P(V_K=k) = \nu_k(K)$ for all $k=0,...,n$.
The following identity was derived in \cite{amelunxen2014living,mccoy2014steiner}:
\begin{equation}
    \delta(K) = \sum_{k=0}^n k \nu_k(K),
    \label{eq:E-V_K}
\end{equation}
that is,
the statistical dimension $\delta(K)$ is the expectation of the
random variable $V_K$.
Furthermore, 
the random variable $V_K$ concentrates around its expected value.
The following concentration inequality is given in \cite[Corollary 4.10]{mccoy2014steiner}
\begin{equation}
    \mathbb P
    \left(
        V_K - \delta(K) \ge \lambda
    \right) 
    \le
    \exp\left(-
    \frac{\delta(K) }{2} 
    h\left( \frac{\lambda}{ \delta(K)} \right)
    \right),
    \qquad 
    \text{for all }
    \lambda>0
    ,
\end{equation}
where $h(t) = (1+t)\log(1+t) - t$.
Using the estimate $h^{-1}(t) \le \sqrt{2t} + 3t$ (cf. \cite[Corollary 12.12]{boucheron2013concentration}),
we obtain
\begin{equation}
    \label{eq:concentration-V}
    \mathbb P
    \left(
        V_K - \delta(K) \ge 2 \sqrt{x \delta(K)} + 6x
    \right) \le \exp(-x)
    ,
    \qquad 
    \text{for all }
    x>0
    .
\end{equation}

Deriving upper and lower bounds on the statistical dimension of a cone $K$ may be a challenging problem.
Some recipes to derive such bounds are proposed in \cite{chandrasekaran2012convex,amelunxen2014living}.
An exact formula is available for the statistical dimension of the cone $\increasings$
\cite[(D.12)]{amelunxen2014living}.
It 
is given 
by
\begin{equation}
\label{eq:delta-increasing}
    \delta(\increasings)
    =
    \delta(\decreasings)
    = \sum_{k=1}^n \frac{1}{k},
    \qquad \text{ so that }\qquad \log n \le \delta(\increasings) \le \log(en).
\end{equation}

Finally, we will need the following characterization of the faces of the cone $\increasings$.
The following proposition may be derived easily from the fact that
if $K$ is the polyhedron \eqref{eq:def-polyhedral},
and a face of $K$ has the form \eqref{eq:face-form}.

\begin{prop}
    \label{prop:faces-increasing}
    Let $d\in\{1,...,n\}$.
    The faces of dimension $d$ of the cone $\increasings$ are the sets
    \begin{equation}
        F(S) \coloneqq \{ \vu = (u_1,...,u_n)^T \in\increasings: u_{i-1} = u_i \text{ if } i\in S\}
    \end{equation}
    where $S\subseteq\{2,...,n\}$ with $|S| = n-d$.
    The cone $\increasings$ has no face of dimension 0.
\end{prop}

Thus, for all $k=1,...,n$, the set $\increasings(k)$ is the union
of all faces of dimension $k$.

\section{Adaptive confidence sets for nondecreasing sequences}
\label{s:honest-inc}

The estimator $\ls(\increasings)$ projects $\vy$ onto $\increasings$,
so the vector $\ls(\increasings)$ is nondecreasing.
Let $\hat k = k(\ls(\increasings)$ be the number of constant pieces of the Least Squares estimator.
Using this notation,
we define the statistic
\begin{equation}
    \label{eq:def-r-inc}
    \radiusInc = \frac{\sigma^2}{n} \left(\sqrt{\hat k \log(en/\hat k)} + \sqrt{2(\log(n/\alpha) + \hat k \log(en/\hat k) )}\right)^2.
\end{equation}

\begin{thm}
    \label{thm:confidence-inc}
    For all $\alpha\in(0,1)$ and all $\vmu\in\increasings$,
    the statistic $\radiusInc$ defined in
    \eqref{eq:def-r-inc} satisfies
    \begin{equation}
        \scalednorms{\ls(\increasings) - \vmu}
        \le \radiusInc
        \label{eq:confidence-inc-honest}
    \end{equation}
    with probability at least $1-\alpha$.
\end{thm}
The above proposition shows that the
confidence set \eqref{eq:Cn-around-ls}
with $\hat r = \radiusInc$ satisfies
condition \eqref{goal:confidence}.
Up to constants and logarithmic factors,
the number of constant pieces $\hat k$ of the Least Squares estimator $\ls(\increasings)$
bounds the loss
$\scalednorms{\ls(\increasings) - \vmu}$ from above with high probability.
Since $\ls(\increasings)$ can be computed in linear time,
the integer $\hat k$ and the statistic $\radiusInc$ can also be computed in linear time. 
It is easy to compute $\hat k$ visually
by drawing the estimator $\ls(\increasings)$ and counting the number of jumps.
The proof of \Cref{thm:confidence-inc} relies on concentration
properties of the random variable \eqref{eq:almost-sure-delta}.

\begin{proof}[Proof of \Cref{thm:confidence-inc}]
    Throughout the proof, we will consider partitions $(T_1,...,T_k)$ of $\{1,...,n\}$
    such that each $T_j$ satisfies $\max T_j < \min T_{j+1}$ as well as
    \begin{equation}
        \label{eq:form-T}
        T_j =\{\min T_j,\min T_j+1,...,\max T_j\},
    \end{equation}
    that is, $T_j$ contains all consecutive integers from $\min T_j$ to $\max T_j$.

    Let $\hmu=\ls(\increasings)$ for notational simplicity.
    Then
    \eqref{eq:strong-convexity-confidence} with $\vu$ replaced
    by $\vmu$ can be rewritten as
    $\euclidnorms{\hmu - \vmu}
        \le
        2 \vxi^T(\hmu-\vmu)
        - \euclidnorms{\hmu-\vmu}$,
        which implies that $\euclidnorm{\hmu - \vmu}\le \vxi^T \hat\vtheta$
    where $\hat\vtheta = (\hmu - \vmu)/\euclidnorm{(\hmu - \vmu)}$
    has Euclidean norm 1.
    By definition of $k(\cdot)$,
    there exists a partition $(\hat T_1,...,\hat T_{\hat k})$ of $\{1,...,n\}$
    such that
    $\ls(\increasings)$ is constant on each $\hat T_j$, $j=1,...,\hat k$.
    Since $\vmu\in\increasings$, both $(\hmu - \vmu)$ and $\hat\vtheta$ belong to the product cone
    \begin{equation}
        \label{def-mathcal-C-hat}
        \mathcal{\hat C} \coloneqq \mathcal S^\downarrow_{|\hat T_1|}\times \mathcal S^\downarrow_{|\hat T_2|}
        \times \dots \times \mathcal S^\downarrow_{|\hat T_{\hat k}|}.
    \end{equation}
    If $\mathcal C$ is of the form $\mathcal C=\mathcal S^\downarrow_{n_1}\times\dots\times\mathcal S^\downarrow_{n_k}$
    for positive integers $n_1,...,n_k$ such that $n_1+...+n_k=n$, then by the Gaussian concentration theorem
    \cite[Theorem 5.5, 5.6]{boucheron2013concentration},
    \begin{equation}
    \label{gaussian-concentration}
        \sup_{\vtheta\in \mathcal C: \euclidnorm{\vtheta}=1} \vxi^T\vtheta \le \E\sup_{\vtheta\in \mathcal C:|\vtheta|_2=1} \vxi^T\vtheta + \sigma\sqrt{2x}
    \end{equation}
    with probability at least $1-e^{-x}$.
    Furthermore, $\E\sup_{\vtheta\in \mathcal C:|\vtheta|_2=1} \vxi^T\vtheta \le \sigma \delta(\mathcal C)^{1/2}$ by Jensen's inequality,
    and $\delta(\mathcal C) = \sum_{j=1}^k \delta(\mathcal S^\downarrow_{n_j})\le \sum_{j=1}^k \log(en_j) \le k \log(en/k)$
    thanks to \eqref{eq:delta-increasing}.

    Let $k=1,...,n$ be fixed. There are ${n-1\choose k}$ partitions of $\{1,...,n\}$ of the form $(T_1,...,T_k)$
    with $\max T_j < \min T_{j+1}$ for all $j=1,...,k-1$ (each such partition defines a unique configuration of jumps).
    By the union bound and inequality $\log {n-1\choose k} \le k \log(en/k)$, we have with probability at least $1-e^{-x}$
    the bound
    $$\sup_{(T_1,...,T_k)}\Big(\sup_{\vtheta\in \mathcal S^\downarrow_{|T_1|}\times \dots \times \mathcal S^\downarrow_{|T_k|}:\euclidnorm{\vtheta}=1}
    \vxi^T\vtheta \Big) \le \sigma \sqrt{k \log(en/k)} + \sigma \sqrt{2(x + k \log(en/k) )}.$$
    Finally, we apply the union bound over all $k\in\{1,...,n\}$ and set $x=\log(n/\alpha)$.
    We have established that with probability at least $1-\alpha$, for any random partition $(\hat T_1,...,\hat T_{\hat k})$
    and $\mathcal{\hat C}$ in \eqref{def-mathcal-C-hat},
    \begin{equation}
        \label{work-for-increasings}
        \sup_{\vtheta\in \mathcal{\hat C}:\euclidnorm{\vtheta}=1}
        \vxi^T\vtheta 
        \le \sigma \sqrt{\hat k \log(en/\hat k)} + \sigma\sqrt{2(\log(n/\alpha) + \hat k \log(en/\hat k) )}.
    \end{equation}
    In particular, this is true for the partition $(\hat T_1,...,\hat T_{\hat k})$ induced by $\ls(\increasings)$ defined in the previous paragraph.
    Note also the right hand side of the previous display is equal to $(n\hat r_\uparrow)^{1/2}$.
    On this event of probability at least $1-\alpha$ we have
    $\euclidnorm{\hmu - \vmu}
        \le
        \vxi^T \hat\vtheta
        \le (n\hat r_\uparrow)^{1/2}
        $ and the proof is complete.
\end{proof}

We have established the existence of an honest confidence interval of the form
\begin{equation}
    \hat C_n \coloneqq
    \{
            \vv\in\increasings:
            \quad
            \scalednorms{\vv - \ls(\increasings)} \le \hat r
    \}.
\end{equation}
This confidence set has uniform coverage over all $\vmu\in\increasings$,
i.e., it satisfies \eqref{goal:confidence}.
The next result implies that the diameter of this confidence set is minimax optimal
up to logarithmic factors.

\begin{thm}
    \label{thm:minimax-k}
    Let $(T_1,...,T_k)$ be a partition of $\{1,...,n\}$
    with $\max T_j < \min T_{j+1}$ for every $j=1,...,k-1$,
    and assume that $\vmu$ is constant on each $T_j$ 
    Let $\gamma\in(0,1)$.
    The random variable $\hat k = k(\ls(\increasings))$
    satisfies
    \begin{align}
        \hat k 
        &\le
        k \log ( en / k )
        +2\sqrt{\log(1/\gamma)\log ( en / k )}
        +6\log(1/\gamma)
        \\
        &\le
        2 k \log ( en / k )
        + 7 \log(1/\gamma)) 
        \label{eq:minimax-k-proba}
    \end{align}
    with probability greater than $1-\gamma$.
    Furthermore,
    \begin{equation}
        \label{eq:minimax-k-E}
        \Evmu[\hat k] \le  k \log ( {en}/{k} ).
    \end{equation}
\end{thm}
Note that in \Cref{thm:minimax-k} $\vmu$ is only assumed to be piecewise constant, but not necessarily  monotone. This will be useful in \Cref{s:unimodal}.
\begin{proof}[Proof of \Cref{thm:minimax-k}]
    Define the closed convex cone 
    \begin{equation}
        K = \mathcal S^\uparrow_{|T_1|} \times \mathcal S^\uparrow_{|T_2|} \times ... \times \mathcal S^\uparrow_{|T_k|}
        \subset \R^n
    \end{equation}
    and let $\hmu^* = \Pi_K(\vy)$. 
    It is clear that
    \begin{equation}
        \min_{\vu\in K} \sum_{j=1}^k \euclidnorms{\vy_{T_j} - \vu_{T_j}}
        =
        \min_{\vu_1\in \mathcal S^\uparrow_{|T_1|},...,\vu_k\in\mathcal S^\uparrow_{|T_k|}}
        \sum_{j=1}^k \euclidnorms{\vy_{T_j} - \vu_{T_j}}
        =
        \sum_{j=1}^k 
        \min_{\vu_j\in \mathcal S^\uparrow_{|T_j|}}
        \euclidnorms{\vy_{T_j} - \vu_{T_j}}.
    \end{equation}
    Thus, as $\vy=\vmu+\vxi$ and $\vmu$ is constant on each $T_j$,
    we have
    \begin{equation}
        \hmu^*_{T_j} 
        = \Pi_{\mathcal S^\uparrow_{|T_j|}}(\vy_{T_j})
        = \Pi_{\mathcal S^\uparrow_{|T_j|}}(\vmu_{T_j} + \vxi_{T_j})
        = \vmu_{T_j} + \Pi_{\mathcal S^\uparrow_{|T_j|}} (\vxi_{T_j})
        .
    \end{equation}
    As adding the constant sequence $\vmu_{T_j}$ does not modify the number of constant pieces (or the number of jumps),
    we have
    \begin{equation}
        k\left(\hmu^*_{T_j}\right) 
        = k\left(\vmu_{T_j} + \Pi_{\mathcal S^\uparrow_{|T_j|}}(\vxi_{T_j}) \right)
        = k\left(\Pi_{\mathcal S^\uparrow_{|T_j|}}(\vxi_{T_j}) \right)
        = k\left(\left(\Pi_K(\vxi)\right)_{T_j} \right).
    \end{equation}
    Let $V_K$ be the random variable defined in \eqref{eq:def-V_K}.
    By the properties of product cones given in
    \cite[Section 5.2]{mccoy2014steiner}, $V_K$ has the same distribution as
    \begin{equation}
        V_{\mathcal S^\uparrow_{|T_1|}} 
        +
        ...
        +
        V_{\mathcal S^\uparrow_{|T_k|}} .
    \end{equation}
    By \Cref{prop:faces-increasing}, for all $j=1,...,k$
    we have $k(\hmu^*_{T_j}) = V_{\mathcal S^\uparrow_{|T_j|}}$ so that
    $V_k$ and $\sum_{j=1}^k k(\hmu^*_{T_j})$ have the same distribution.
    By \eqref{eq:E-V_K}, $\E V_K = \delta(K)$
    and by \eqref{eq:concentration-V}, with probability greater than $1-\gamma$
    we have
    \begin{align}
        V_K
        &\le
        \delta(K)
        +2\sqrt{\log(1/\gamma)\delta(K)}
        + 6 \log(1/\gamma).
    \end{align}
    To bound $\delta(K)$ from above, we use that the statistical dimension
    of a direct product of cones is the sum of the statistical dimensions (cf. \cite[Proposition 3.1]{amelunxen2014living})
    \begin{equation}
        \delta(K)
        = \sum_{j=1}^k \delta(\mathcal S^\uparrow_{|T_j|})
        \le 
        \sum_{j=1}^k \log(e|T_j|)
        \le
        k \log(en/k),
    \end{equation}
    where we have used \eqref{eq:delta-increasing}
    and Jensen's inequality for the two inequalities.

    The random variable $V_K$ is distributed as $\sum_{j=1}^k k(\hmu^*_{T_j})$.
    Thus, to complete the proof, it is enough to prove that almost surely, $\hat k \coloneqq k(\ls(\increasings)) \le \sum_{j=1}^k k(\hmu^*_{T_j})$.
    Let $\hmu=\ls(\increasings)$ for notational simplicity.
    It is clear that
    \begin{equation}
        k(\hmu) = | \{ \hat\mu_i, i=1,...,n\} |
        \le
        \sum_{j=1}^k k(\hmu_{T_j})
        =
        \sum_{j=1}^k 
        | \{ \hat\mu_i, i\in T_j\} |
        ,
    \end{equation}
    since a piece counted on the left hand side must be counted at least once (and possibly multiple times) on the right hand side.
    For all $j=1,...,k$, the vectors $\hmu_{T_j}$ and $\hmu^*_{T_j}$ are solutions of the minimization problems
    \begin{equation}
        \hmu^*_{T_j}
        =
        \argmin_{\vv\in\mathcal S^\uparrow_{|T_j|}}\euclidnorms{\vv - \vy_{T_j}},
        \qquad
        \hmu_{T_j}
        =
        \argmin_{
            \substack{
                \vv\in\mathcal S^\uparrow_{|T_j|}: \\
                \hmu_{\min(T_j)}\le \vv_1,\\
                \vv_{|T_j|} \le \hmu_{\max(T_j)}
            }
        }
        \euclidnorms{\vv - \vy_{T_j}}.
    \end{equation}
    This means that $\hmu_{T_j}$ is solution of a minimization problem with additional constraints at the boundary.
    By \Cref{lemma:number-of-jumps},
    we have
    \begin{equation}
        k\left(\hmu_{T_j}\right) \le k\left(\hmu^*_{T_j}\right)
    \end{equation}
    for all $j=1,...,k$, which completes the proof.
\end{proof}

\begin{cor}
    Let $J=\{1,...,n\}$ and define the collection of models
    $(E_k)_{k \in J} = \left(\increasings(k)\right)_{k\in J}$.
    The random variable $\radiusInc$ defined in \eqref{eq:def-r-inc}
    satisfies 
    \eqref{eq:confidence-inc-honest},
    \eqref{goal:minimax-ball}
    and \eqref{goal:minimax-ball-E} with $\hat r$ replaced by $\radiusInc$.
    Thus, the ball centered at $\ls(\increasings)$ of radius $\sqrt{\radiusInc}$
    is an honest confidence set,
    which is adaptive in probability and
    in expectation
    with respect to the models $\left(\increasings(k)\right)_{k=1,...,n}$.
\end{cor}

\section{Adaptive confidence sets for convex sequences}
\label{s:honest-conv}

Confidence sets can also be obtained
in univariate convex regression.
If $n\ge 3$, define the set of convex sequences $\convexs$ by
\begin{align}
    \convexs 
    & \coloneqq \{ \vu=(u_1,\dots,u_n)^T\in\Rn: \; 2u_i \le u_{i+1} + u_{i-1}, \  i=2,\dots,n-1 \},
\end{align}
and define $\convexs = \R$ if $n=1$ and $\convexs = \R^2$ if $n=2$.
For all $n\ge1$, define the cone of concave sequences by $\concaves \coloneqq - \convexs$.

For any $\vu\in\convexs$, let $q(\vu) - 1\ge 0$ be 
the number of inequalities 
$2u_i \le u_{i+1} + u_{i-1}, i=2,...,n-1$ that are strict.
The integer $q(\vu)$ is also the smallest positive integer
such that $\vu$ is piecewise affine with $q(\vu)$ pieces.
Define the sets
\begin{equation}
    \convexs(q) = \{ \vu\in\convexs: q(\vu) \le q \},
    \qquad
    q=1,...,n-1.
\end{equation}
The set $\convexs(1)$ is the subspace of all affine sequences
while $\convexs(2),...,\convexs(n-2)$ are closed non-convex sets.
We have
\begin{equation}
    \convexs(1)
    \subset
    \convexs(2)
    \subset
    ...
    \subset
    \convexs(n-1)
    =
    \convexs.
\end{equation}
These sets represent models of increasing complexity.

There exist numerical constants $c,c'>0$ such that
for all $\alpha\le (0,\min(c,1))$ and
any $q=1,...,n-1$, we have
\begin{equation}
    \frac{c'\sigma^2 q}{n}
    \le
        R^*_\alpha(\convexs(q)) 
    \le
    \frac{20 \sigma^2 q \log(en/q)}{n}
    + \frac{10\log(1/\alpha)}{n},
    \label{eq:R_alpha_convex}
\end{equation}
cf. \cite[Theorem 6]{bellec2015shape} for the upper bound
and \cite[Proposition 7]{bellec2015sharp} for the lower bound.
Thus, for $\alpha>0$ small enough, the quantity
$R^*_\alpha(\convexs(q))$ is of order $q\sigma^2/n$, up to logarithmic factors.

The statistical dimension of the cone  $\convexs$ satisfies
\cite{bellec2015shape}
\begin{equation}
    \label{eq:delta-convex}
    \delta(\convexs)
    =
    \delta(\concaves)
    \le 10 \log(en).
\end{equation}
It is not known whether this upper bound is sharp.
However, the fact that the statistical dimension of $\convexs$ grows slower
that a logarithmic function of $n$ is enough for the purpose of the present paper.

The following bound on the risk of $\ls(\convexs)$ will be useful.
\begin{prop}[\cite{bellec2015shape}]
    \label{prop:convexs}
    Let $\vmu\in\convexs$.
    Then
    \begin{equation}
        \Evmu \euclidnorms{\ls(\convexs) - \vmu}
        \le
        \Evmu \left[ ( \sup_{\vv\in \mathcal T_{\vmu}: \euclidnorm{\vv}\le 1 } \vxi^T \vv )^2 \right]
        = \sigma^2\delta(\mathcal T_{\vmu})
        \le 10 \sigma^2 q(\vmu)\log \frac{en}{q(\vmu)},
    \end{equation}
    where $\mathcal T_{\vmu}$ is the tangent cone at $\vmu$ defined by
    $\mathcal T_{\vmu} \coloneqq
        \{ \vu - t \vmu, \vu\in\convexs,t\in\R \}
    $.
\end{prop}
An outline of the proof of this result is as follows. More details may be found in \cite{bellec2015shape}.
\begin{proof}[Outline of the proof of \Cref{prop:convexs}]
    The inequality
    $\Evmu \euclidnorms{\ls(\convexs) - \vmu}
    \le
    \sigma^2\delta(\mathcal T_{\vmu})$
    was proved by \cite{oymak2013sharp},
    it is a direct consequence of \eqref{eq:strong-convexity-confidence} with $\vu=\vmu$.
    To bound from above the statistical dimension of $\mathcal T_{\vmu}$,
    we have the inclusion
    \begin{equation}
        \mathcal T_{\vmu}
        \subset
        \mathcal S^\cup_{|T_1|}
        \times
        ...
        \times
        \mathcal S^\cup_{|T_q(\vmu)|}
        ,
    \end{equation}
    where $(T_1,...,T_{q(\vmu)})$ is a partition of $\{1,...,n\}$ such 
    that $\vmu$ is affine on each $T_j$, $j=1,...,q(\vmu)$.
    The formula for the statistical dimension of a direct product of cones \cite[Proposition 3.1]{amelunxen2014living} yields
    \begin{equation}
        \delta(
            \mathcal S^\cup_{|T_1|}
            \times
            ...
            \times
            \mathcal S^\cup_{|T_q(\vmu)|}
        )
        = \sum_{j=1}^{q(\vmu)}
            \delta(\mathcal S^\cup_{|T_j|})
        \le 10 \sum_{j=1}^{q(\vmu)} \log(e|T_j|)
        \le 10 \log(en/q(\vmu)),
    \end{equation}
    where we used \eqref{eq:delta-convex} and Jensen's inequality.
\end{proof}

We now turn to the construction of confidence sets.
Recall that if $\vu\in\convexs$
is a convex sequence,
$q(\vu)$ is the number of pieces in the piecewise affine decomposition
of $\vu$.
Let $\hat q \coloneqq q(\ls(\convexs)$
be the number of affine pieces of the Least Squares estimator.
Then, define the statistic
\begin{equation}
    \label{eq:def-r-conv}
    \radiusConv= \frac{\sigma^2 \hat q ( 20 + 40 \log(n) + 10 \log(1/\alpha))}{n}.
\end{equation}
Similarly to the case of the statistic $\radiusInc$ in isotonic regression,
the following result shows that
the confidence ball \eqref{eq:Cn-around-ls}
with $\hat r = \radiusConv$ enjoys uniform coverage
over all $\vmu\in\convexs$.

\begin{thm}
    \label{thm:confidence-conv}
    For all $\alpha\in(0,1)$ and all $\vmu\in\convexs$,
    the statistic $\radiusConv$ defined in
    \eqref{eq:def-r-inc} satisfies
    \begin{equation}
        \scalednorms{\ls(\convexs) - \vmu}
        \le \radiusConv
        \label{eq:confidence-conv-honest},
    \end{equation}
    with probability at least $1-\alpha$.
\end{thm}
The above result is analog to \Cref{thm:confidence-inc}.
The numerical constants are slightly worse in the case of the present section
because the upper bound \eqref{eq:delta-convex}
on the statistical dimension of the cone $\convexs$
is slightly worse than \eqref{eq:delta-increasing}.
The proof of \Cref{thm:confidence-conv}
is similar to the proof of \Cref{thm:confidence-inc}
and can be found in \Cref{appendix:convex}.

Now, the goal is to show that the statistic $\radiusConv$
is of the same order as the minimax quantity \eqref{eq:R_alpha_convex}.
We employ a different strategy than in the previous section.
For any function $g:\R^n\rightarrow\R^n$ which is weakly
differentiable,
the divergence of $g$ is the random variable
\begin{equation}
    D_g(\vy) = \sigma^2 \sum_{i=1}^n \frac{\partial}{\partial y_i}  g(\vy)_i
    .
\end{equation}
It is well known that by Stein's identity,
under suitable conditions on $g$
(cf. \cite[Section 2]{meyer2000degrees} or \cite[Lemma 3.6]{tsybakov2009introduction}),
we have
\begin{equation}
    \sigma^2 \Evmu D_g(\vy) = \Evmu[\vxi^Tg(\vy)].
    \label{eq:stein}
\end{equation}
The divergence of the estimator $\ls(\convexs)=\Pi_{\convexs}(\vy)$ is given in \cite[Proposition 2.7]{chen2015degrees}
(see also \cite{meyer2000degrees}).
Namely, we have the following result.
\begin{prop}[\cite{meyer2000degrees,chen2015degrees}]
    \label{prop:sen}
    If $g(\cdot)=\Pi_{\convexs}(\cdot)$ is the projection onto the cone of convex sequences,
    then \eqref{eq:stein} holds and we have
    \begin{equation}
        D_g(\vy) = \hat q + 1
    \end{equation}
    almost surely,
    where $\hat q = q(\ls(\convexs))$.
\end{prop}
This result can be used to bound from above the expected radius of the statistic $\radiusConv$.

\begin{thm}
    \label{thm:minimax-q-E}
    Let $\vmu\in\convexs$. Then
    \begin{equation}
        \Evmu [\hat q] \le 10 q(\vmu) \log\frac{en}{q(\vmu)} -1.
        \label{eq:expected-q-hat}
    \end{equation}
    Consequently, for all $\alpha\in(0,1)$, the statistic \eqref{eq:def-r-conv} satisfies
    \begin{equation}
        \Evmu[ \radiusConv ]
        \le
        \frac{\sigma^2 q(\vmu) \polylog(n,1/\alpha)}{n}.
        \label{eq:expected-radius-conv}
    \end{equation}
    where $\polylog(n,1/\alpha) = 10 \log(en) (20 + 40 \log(n) + 10 \log(1/\alpha))$.
\end{thm}
\begin{proof}
    By \Cref{prop:sen} and \eqref{eq:stein}, we have
    \begin{equation}
        \sigma^2 \Evmu [1 + \hat q]
        =
        \Evmu[ \vxi^T \Pi_{\convexs}(\vy) ]
        =
        \Evmu[ \vxi^T (\Pi_{\convexs}(\vy) - \vmu) ].
    \end{equation}
    By the Cauchy-Schwarz inequality, we have
    \begin{align}
        \sigma^2 \Evmu [1 + \hat q]
        &\le
        \Evmu^{1/2} \left[ \left( \frac{\vxi^T (\Pi_{\convexs}(\vy) - \vmu)}{\euclidnorm{\Pi_{\convexs}(\vy) - \vmu}} \right)^2 \right]
        \Evmu^{1/2} \euclidnorms{\Pi_{\convexs}(\vy) - \vmu},
        \\
        &\le \sigma \sqrt{ \delta(\mathcal T_{\vmu}) }
        \Evmu^{1/2} \euclidnorms{\Pi_{\convexs}(\vy) - \vmu}.
    \end{align}
    Using \Cref{prop:convexs} completes the proof of \eqref{eq:expected-q-hat}.
    Inequality \eqref{eq:expected-radius-conv}
    is a direct consequence of \eqref{eq:expected-q-hat}
    and of the definition of $\radiusConv$.
\end{proof}
The above result is different from \Cref{thm:minimax-k}
in isotonic regression.
\Cref{thm:minimax-k} controls both the expectation and the deviations of $\hat k$.
In this section,
\Cref{thm:minimax-q-E} only controls the expectation of $\hat q$.
This comes from the use of Stein's identity in the proof of \Cref{thm:minimax-q-E},
which yields a result only in expectation.

The arguments used to prove \Cref{thm:minimax-k}
are based on the concentration properties of the intrinsic volumes of cones,
while the proof of \Cref{thm:minimax-q-E} relies 
on Stein's identity and \Cref{prop:sen}.
Thus, we have presented two methods to bound from above
the expected diameter of the confidence sets constructed in the present paper.

\begin{cor}
    Let $J=\{1,...,n-1\}$ and define the collection of models
    $(E_k)_{k\in J} = \left(\convexs(k)\right)_{k=1,...,n-1}$.
    The random variable $\radiusConv$ defined in \eqref{eq:def-r-conv}
    satisfies 
    \eqref{eq:confidence-conv-honest}
    and \eqref{goal:minimax-ball-E} with $\hat r$ replaced by $\radiusConv$.
    Thus, the ball centered at $\ls(\convexs)$ of radius $\sqrt{\radiusConv}$
    is an honest confidence set,
    which is adaptive
    in expectation
    with respect to the models $\left(\convexs(k)\right)_{k=1,...,n-1}$.
\end{cor}

\section{Non-convexity: Adaptive confidence sets for unimodal sequences}
\label{s:unimodal}

Let $m\in\{1,...,n\}$.
A sequence $\vu\in\R^n$
is unimodal with mode at position $m$
if and only if $\vu_{\{1,...,m\}}$ is non-increasing
and $\vu_{\{m,...,n\}}$ is nondecreasing--in other words, $\vu$ belongs to the set
\begin{equation}
    K_m \coloneqq \{ \vu=(u_1,...,u_n)^T\in\R^n: u_1 \ge ... \ge u_m \le u_{m-1} \le ... \le u_n \}.
    \label{eq:def-Km-unimodal}
\end{equation}
Next, by taking the union over all possible locations for the mode, we define the set of all unimodal sequences as
\begin{equation}
    \mathcal U \coloneqq \cup_{m=1,...,n} K_m
    .
    \label{eq:def-U-unimodal}
\end{equation}
The set $\mathcal U$ is non-convex for $n\ge 3$.

Recall that if $\vu\in\increasings\cup\decreasings$ is monotone, $k(\vu)=|\{u_i, i=1,...,n\}$ is the smallest positive integer $k$ such that
$\vu$ is piecewise constant constant with $k$ pieces.
We extend the function $k$ to the set of unimodal sequences by setting
\begin{multline}
\forall \vu\in\mathcal U, \quad k(\vu) = \min\{ k\ge 1: 
\exists \text{ partition } (T_1,...,T_k) : \\ \vu_{T_j} \text{ is constant for all } j=1,...,k \},
\end{multline}
where the partition $(T_1,...,T_k)$ is a partition of $\{1,...,n\}$ with $\max T_j < \min T_{j+1}$ for all $j=1,...,k-1$.
A unimodal sequence $\vu\in\mathcal U$ has $k(\vu)-1$ jumps.
Similarly to isotonic regression, define the models
\begin{equation}
    \mathcal U(k) = \{ \vu\in\mathcal U:  k(\vu) \le k \},
    \qquad
    k=1,...,n.
    \label{models-unimodal}
\end{equation}
These sets define models of increasing complexity since
\begin{equation}
    \mathcal U(1)
    \subset
    \mathcal U(2)
    \subset
    ...
    \subset
    \mathcal U(n)
    =
    \mathcal U
\end{equation}
It is known 
that there exist numerical constants $c,c',c''$ such that for all $\alpha \le c$,
\begin{equation}
    c'\sigma^2k / n
    \le R^*_\alpha(\mathcal U(k)) 
    \le
    c''\sigma^2(k \log(en/k) + \log(n/\alpha) ) / n.
\end{equation}
Indeed, the lower bound is a consequence of \eqref{eq:minimax-loss-increasings}
and the inclusion $R^*_\alpha(\increasings(k))\subset R^*_\alpha(\mathcal U(k))$,
while the upper bound is proved in 
\cite{bellec2015shape,chatterjee2015adaptive,flammation2016seriation},
see \cite[Appendix C ]{bellec2015shape} for explicit constants.
Thus, for $\alpha>0$ small enough, the quantity
$R^*_\alpha(\increasings(k))$ is of order $k\sigma^2/n$,
up to logarithmic factors in $n$ and $1/\alpha$.
Similarly,
the minimax risk over the sets $\increasings(k)$ satisfies
\begin{equation}
    c'\sigma^2 k /n
    \le
    R^*_{\E}(\mathcal U(k)) 
    \le
    \sup_{\vmu\in\mathcal U(k)} \Evmu \scalednorms{\ls(\mathcal U) - \vmu} 
    \le
    c'' \sigma^2 k \log(en/k) / n,
\end{equation}
for some numerical constants $c',c''>0$ for instance by integrating the bounds of \cite[Appendix C]{bellec2015shape},
where $\ls(\mathcal U)\in \argmin_{\vu\in\mathcal U} \|\vu - \vy\|$ is the non-convex unimodal Least-Squares estimator
(if the non-convex optimization problem has several solutions, we break ties arbitrarily).

The results below show that, using the number of constant pieces $\hat k=p(\ls(\mathcal U))$ of the unimodal Least-Squares estimator,
one can construct adaptive confidence sets with respect to the collection of models \eqref{models-unimodal}.

\begin{thm}
    \label{thm:confidence-uni}
    Let $\hat p = k(\ls(\mathcal U))$ be the smallest integer such that $\ls(\mathcal U)$ is piecewise constant on $\hat p$ contiguous pieces.
    For all $\alpha\in(0,1)$ and all $\mu\in\mathcal U$, the statistic 
    $$\hat r_{uni}\coloneqq \frac{\sigma^2}{n} 
    \left(\sqrt{(\hat p+3) \log\left(\frac{en}{\hat p+3}\right)} + \sqrt{2\log\left(\frac{n^2}{\alpha}\right) + 2(\hat p + 3) \log\left(\frac{en}{\hat p + 3} \right)  }\right)^2
        $$
    satisfies
    $\scalednorms{\ls(\mathcal U)-\vmu} \le \hat r_{uni}$
    with probability at least $1-\alpha$.
\end{thm}

\begin{proof}
    Let $\hmu=\ls(\mathcal U)$ for brevity.
    By the optimality conditions of $\hmu$, we have $\euclidnorms{\hmu-\vmu}\le 2\vxi^T(\hmu - \vmu)$
    and 
    \begin{equation}
        \euclidnorm{\hmu-\vmu}\le 2\vxi^T\hat\vtheta \qquad \text{ where } \qquad \hat\vtheta = (\hmu-\vmu)/\euclidnorm{\hmu-\vmu}.
        \label{unimodal-intermediate}
    \end{equation}
    We split $\{1,...,n\}$ into a partition $(L,C,R)$ as follows: $L$ is the largest set of indices of the form
    $\{1,2,...,|L|\}$ such that both $\hmu_L$ and $\vmu_L$ are non-increasing,
    $R$ is the largest set of indices of the form $\{n-|R|+1,n-|R|+2...,n\}$ such that $\hmu_R$ and $\vmu_R$ are both non-decreasing,
    and $C=\{1,...,n\}\setminus(R\cup L)$ contains the remaining central indices, where by definition, $\hmu_C-\vmu_C$ is either increasing or decreasing.
    Let $\hat k_L$ be the number of constant pieces of $\hmu$ on $L$, and let 
    $(\hat S_1,...\hat S_{\hat k_L})$ be a partition of $L$ such that $\hmu_L$ is constant on each $\hat S_j$.
    Similarly, let $\hat k_R$ be the number of constant pieces of $\hmu$ on $R$
    and $(\hat T_1,...,\hat T_{\hat k_R})$ be a partition of $R$ such that $\hmu_R$ is constant on each $\hat T_j$.
    Then both $\hmu-\vmu$ and $\hat\vtheta$ belong to either one of the cones
    $$
        \mathcal{\hat D_+} \coloneqq 
        \left(\mathcal S^\uparrow_{|\hat S_1|}
        \times \dots \times \mathcal S^\uparrow_{|\hat S_{\hat k_L}|}\right)
        \times \mathcal S^\uparrow_{|C|} \times
        \left(\mathcal S^\downarrow_{|\hat T_1|}
        \times \dots \times \mathcal S^\downarrow_{|\hat T_{\hat k_R}|}\right).
    $$
    or
    $$
        \mathcal{\hat D_-} \coloneqq 
        \left(\mathcal S^\uparrow_{|\hat S_1|}
        \times \dots \times \mathcal S^\uparrow_{|\hat S_{\hat k_L}|}\right)
        \times \mathcal S^\downarrow_{|C|} \times
        \left(\mathcal S^\downarrow_{|\hat T_1|}
        \times \dots \times \mathcal S^\downarrow_{|\hat T_{\hat k_R}|}\right),
    $$
    the only difference being the direction on the central indices in $C$.
    We now argue similarly to \eqref{work-for-increasings} as follows.
    If the cone $\mathcal C$ is nonrandom, of the form 
    \begin{equation}
        \label{form-cones}
        \mathcal C=
        \mathcal S^\uparrow_{n_1}\times\dots\times\mathcal S^\uparrow_{n_k}
        \times
        \mathcal S^\downarrow_{m_1}\times\dots\times\mathcal S^\downarrow_{m_l}
    \end{equation}
    for positive integers $n_1,...,n_k,m_1,...,m_l$ such that $n_1+...+n_k+m_1+...+m_l=n$, then by the Gaussian concentration theorem,
    \eqref{gaussian-concentration} holds with probability at least $1-e^{-x}$.
    Furthermore, $\E\sup_{\vtheta\in \mathcal C:|\vtheta|_2=1} \vxi^T\vtheta \le \sigma \delta(\mathcal C)^{1/2}$ by Jensen's inequality,
    and $\delta(\mathcal C) = \sum_{j=1}^k \log(en_j) + \sum_{j=1}^l \log(em_j) \le (k+l) \log(en/(k+l)$
    thanks to \eqref{eq:delta-increasing} and the fact that the statistical
    dimension of a product of cones is equal to the sum of the statistical dimensions
    (cf. \cite[Proposition 3.1]{amelunxen2014living}).

    Let $s=1,...,n$ be fixed. There are fewer than $n {n-1\choose s }$
    cones of the form \eqref{form-cones} with $k+l=s$.
    By the union bound and inequality $\log {n-1\choose s} \le s \log(en/s)$, we have with probability at least $1-e^{-x}$
    the bound
    $$\sup_{\mathcal C}\Big(\sup_{\vtheta\in\mathcal C :\euclidnorm{\vtheta}=1}
    \vxi^T\vtheta \Big) \le \sigma \sqrt{s \log(en/s)} + \sigma \sqrt{2(x +\log n +  s \log(en/s) )}$$
    where the supremum is taken over all cones $\mathcal C$ of the form \eqref{form-cones}
    with $k+l=s$.
    Finally, we apply the union bound over all $s\in\{1,...,n\}$ and set $x=\log(n/\alpha)$.
    We have established that with probability at least $1-\alpha$, 
    $$
    \sup_{\vtheta\in \mathcal{\hat D_+} \cup \mathcal{\hat D_-} : |\vtheta|_2=1 }
    \vxi^T\vtheta
    \le
    \sigma \sqrt{\hat s \log(en/\hat s)} + \sigma \sqrt{2(\log(n^2/\alpha) +  \hat s \log(en/\hat s) )}.$$
    where $\hat s = \hat k_L + \hat k_R + 1 $.
    The proof is completed by combining this bound with \eqref{unimodal-intermediate},
    and observing that $\hat s = \hat k_R + 1 + \hat k_L \le \hat p + 3$.
\end{proof}

Finally, the following result shows that the confidence set of the previous theorem has optimal radius
up to logarithm factors.

\begin{thm}
    \label{thm:minimax-uni}
    If $\vmu\in\mathcal U$ then
    $$\hat p \coloneqq k(\ls(\mathcal U)) \le 4 k(\vmu)\log(en/k(\vmu))+14\log(2n/\gamma)$$
    with probability at least $1-\gamma$.
    Furthermore, $\E[\hat p]\le 4 k(\vmu)\log(en/k(\vmu))+14\log(2 e n )$.
\end{thm}

\begin{proof}
    If $\hmu=\ls(\mathcal U)$ is a unimodal fit and $\hat m$ is the location of its mode, then
    the following facts hold true \cite{stout2008unimodal}:
    (a) the value of $\hat\mu_{\hat m}$ of $\hmu$ at $\hat m$ is equal to $y_{\hat m}$,
    (b) $\hmu_{\{1,...,\hat m\}}$ is equal to the isotonic (decreasing) fit of $\vy_{\{1,...,\hat m\}}$,
    and (c) $\hmu_{\{\hat m,...,n\}}$ is equal to the isotonic (increasing) fit of $\vy_{\{\hat m,...,n\}}$.
    Hence we can bound from above the number of constant pieces of $\hmu$ by the number of constant pieces
    of two isotonic fits, one on $\{1,...,\hat m\}$ and the other on $\{\hat m,...,n\}$.

    If $m$ is a deterministic mode location, and $\vmu_{\{1,...,m\}}$ has $k(\vmu_{\{1,...,m\}})$ pieces,
    by \Cref{thm:minimax-k} the isotonic (decreasing) fit of $\vy_{\{1,...,m}\}$
    has at most $$2k(\vmu_{\{1,...,m\} })\log(em/k(\vmu_{\{1,...,m\} }))+7\log(1/\gamma)$$
    constant pieces with probability $1-\gamma$. Similarly, with at least probability $1-\gamma$,
    the isotonic (increasing) fit of $\vy_{\{m,...,n\}}$ has 
    at most 
    $$2k(\vmu_{\{m,...,n\} })\log(e(n-m+1)/k(\vmu_{\{m,...,n\} }))+7\log(1/\gamma)$$
    constant pieces with probability at lesat $1-\gamma$.
    By the union bound, the two previous sentences hold uniformly over all possible modes $m=1,...,n$
    with probability at least $1-2n\gamma$.

    Hence, with probability at least $1-2n\gamma$, the number of constant pieces of the unimodal
    least-squares $\hmu$ is bounded from above by
    $$2k\left(\vmu_{\{1,...,\hat m\} }\right)\log\left(\frac{e\hat m}{k(\vmu_{\{1,...,\hat m\} })}\right)+ 
        2k\left(\vmu_{\{\hat m,...,n\} }\right)\log\left(\frac{e(n-\hat m+1)}{k(\vmu_{\{\hat m,...,n\} })}\right)
        +14\log\frac 1 \gamma .
    $$
    We first use that $\hat m\le n$ and $(n-\hat m + 1)\le n$ to bound from above the numerator inside the logarithms.
    Next, we use $k(\vmu_{\{1,...,\hat m\} }) \le k(\vmu)$ and the fact that $x \log(en/x)$ is increasing on $[1,n]$
    to conclude that the previous display is bounded from above by
    $4 k(\vmu)\log(en/k(\vmu))+14\log(1/\gamma)$.
    The result in expectation is obtained by integration, using the identity $\E[Z]=\int_0^\infty \mathbb P(Z>t)dt$
    for every non-negative random variable $Z$.

\end{proof}

\section{Concluding remarks}

We have provided a simple construction of honest and adaptive confidence sets
for isotonic, convex and unimodal regression.
Our construction reveals that the complexity of the Least-Squares estimator
in these problems, e.g. the number of jumps of the isotonic Least-Squares
or the number of changes of slope of the convex Least-Squares,
can be used to bound from above the error of the estimator
(cf. \Cref{thm:confidence-inc,thm:confidence-conv,thm:confidence-uni}).
Furthermore, the complexity of the Least-Squares estimator is these problems
is not larger, up to logarithmic factors, than the complexity of the true mean vector
(cf. \Cref{thm:minimax-k,thm:minimax-q-E,thm:minimax-uni}).

The construction of honest confidence sets in \Cref{thm:confidence-inc,thm:confidence-conv,thm:confidence-uni}
relies on a careful application of the Gaussian concentration theorem
combined with union bounds and upper bounds on statistical dimensions of tangent cones.
Such techniques can readily be extended to the setting of
\cite{chatterjee2015matrix,guntuboyina2017spatial,han2017isotonic},
where bounds on the statistical dimensions of tangent cones are readily available.
However, the techniques used in \Cref{thm:minimax-k,thm:minimax-q-E,thm:minimax-uni}
to control the size of such confidence sets do not directly extend to the settings
considered in \cite{chatterjee2015matrix,guntuboyina2017spatial,han2017isotonic}
and it is unclear at this point how to control adaptively the radius of the confidence sets
in these settings.

\appendix

\section{Nondecreasing sequences with bounded total variation}
\label{s:uniform-inc}

Let $V>0$.
If the unknown parameter $\vmu$ satisfies
$\mu_n - \mu_1\le V$, the risk of the Least Squares estimator 
satisfy \cite[(28)]{zhang2002risk}
\begin{equation}
    \Evmu \scalednorms{\ls(\increasings) - \vmu}
    \le
    \sigma^2
    \kappa^2
    \left(
        \left(
            \frac{V}{\sigma n}
        \right)^{2/3}
        + \frac{\log(e n)}{n}
    \right),
\end{equation}
where $\kappa \le 3.6$.
Thus, an explicit constant is readily available \cite[(2.8)]{zhang2002risk}.

In this section, we explain how to construct confidence sets
with diameter of the same order as the right hand side of the previous display.
We proceed as follows.

The function $f:\R^n \rightarrow \R^n$ defined by $f(\vv) = \scalednorm{\Pi_\increasings(\vmu + \sigma \vv) - \vmu}$
is Lipschitz with coefficient $\sigma/\sqrt n$
as for all $\vv,\vv'\in\R^n$,
\begin{equation}
    |f(\vv) - f(\vv')|
    \le
    \scalednorm{\Pi_\increasings(\vmu + \sigma \vv) - \Pi_\increasings(\vmu +\sigma \vv')}
    \le
    \sigma \scalednorm{\vv - \vv'}
    = (\sigma/\sqrt n) \euclidnorm{\vv - \vv'}.
    \label{eq:risk-zhang}
\end{equation}
By the Gaussian concentration inequality \cite[Theorem 5.6]{boucheron2013concentration},
the following holds with probability greater than $1-\alpha$
\begin{equation}
    \scalednorm{\ls(\increasings)-\vmu}
    \le
    \Evmu \scalednorm{\ls(\increasings)-\vmu}
    + \sigma \sqrt{\frac{2\log(1/\alpha)}{n}}.
\end{equation}
Using that $(a+b)^2\le 2a^2+2b^2$,
we obtain the following for all $\alpha\in(0,1)$:
If $\vmu\in\increasings$
and $\mu_n - \mu_1 \le V$, then
\begin{equation}
    \scalednorms{\ls(\increasings)-\vmu}
    \le
    2\kappa^2 \sigma^2
        \left(
            \frac{V}{\sigma n}
        \right)^{2/3}
    +  \frac{
        2\kappa^2\sigma^2\log(e n)
        + 4\sigma^2 \log(1/\alpha)
    }{n}
\end{equation}
with probability greater than $1-\alpha$.

Let $V_\vmu=\mu_n - \mu_1$ and $\hat V = y_n - y_1$. 
The random variable $\hat V - V_\vmu$ is centered Gaussian with variance $2\sigma^2$, so
that
\begin{equation}
    V_\vmu \le \hat V +  2 \sigma \sqrt{\log(1/\alpha)}
\end{equation}
with probability greater than $1-\alpha$.
Thus, we have established the following.
\begin{prop}
    Let $\vmu\in\increasings$.
    Define the statistic $\radiusUniform$ by
    \begin{equation}
        \sqrt{\radiusUniform}
        =
        2\kappa^2 \sigma^2
            \left(
                \frac{\hat V + 2 \sigma\sqrt{\log(1/\alpha)}}{\sigma n}
            \right)^{2/3}
        +  \frac{
            2\kappa^2\sigma^2\log(e n)
            + 4\sigma^2 \log(1/\alpha)
        }{n}
    \end{equation}
    where $\kappa\le 3.6$ is the constant from \cite{zhang2002risk} that appears in \eqref{eq:risk-zhang}.
    Then we have $\scalednorms{\ls(\increasings) - \vmu}\le\radiusUniform$
    with probability greater than $1-\alpha$.
\end{prop}
Furthermore, it is clear that $\hat V \le V_\vmu + 2\sigma\sqrt{\log(1/\gamma)}$
with probability greater than $1-\gamma$ for all $\gamma\in(0,1)$.
\begin{prop}
    Let $\vmu\in\increasings$ and let $V=\mu_n - \mu_1$.
    Then  the statistic  $\radiusUniform$ defined above satisfies
    \begin{equation}
        \radiusUniform
        \le
        2\kappa^2 \sigma^2
            \left(
                \frac{\hat V + 2 \sigma\sqrt{\log(1/(\gamma\alpha))}}{\sigma n}
            \right)^{2/3}
        +  \frac{
            2\kappa^2\sigma^2\log(e n)
            + 4\sigma^2 \log(1/\alpha)
        }{n}
        \label{eq:rhs-radiusUniform}
    \end{equation}
    with probability at least $1-\gamma$ for all $\gamma\in(0,1)$.
\end{prop}

\begin{thm}
    \label{thm:combined}
    Let $\vmu\in\increasings$.
    The statistic $\min(\radiusInc, \radiusUniform)$ 
    satisfies 
    \begin{equation}
        \scalednorms{\ls(\increasings)-\vmu} \le \min(\radiusInc, \radiusUniform)
    \end{equation}
    with probability at least $1- 2 \alpha$.
    Furthermore, for all $\gamma\in(0,1)$,
    the statistic $\min(\radiusInc, \radiusUniform)$ 
    is bounded from above 
    with probability at least $1-2\gamma$,
    by the minimum of the right hand side of \eqref{eq:minimax-k-proba}
    and the right hand side of \eqref{eq:rhs-radiusUniform}.
\end{thm}
For all $V\ge\sigma$ and all $k=1,...,n$, define the class
\begin{equation}
    \increasings(k,V) \coloneqq \{ \vv=(v_1,...,v_n)^T
    \in\increasings: \quad
    k(\vv)\le k \text{ and } v_n - v_1\le V 
    \}.
\end{equation}
For small enough $\alpha>0$,
the quantity $R^*_\alpha(\increasings(k,V))$ defined in \eqref{eq:minimax-loss}
is greater than 
\begin{equation}
    c \sigma^2 
    \min
    \left(
        \left(
            \frac{V}{\sigma n}
        \right)^{2/3}
        ,
        \frac k n
    \right)
\end{equation}
for some absolute constant $c>0$, cf. \cite[Proposition 4]{bellec2015shape}.
Thus, the statistic $\min(\radiusInc, \radiusUniform)$ of \Cref{thm:combined}
induces an honest confidence ball centered at the Least Squares estimator,
and this confidence ball is adaptive in probability
for the collection of models
\begin{equation}
    (\increasings(k,V))_{k\in\{1,...,n\}, V\in [\sigma,+\infty)}.
\end{equation}

\section{Technical Lemma}

\begin{lemma}
    \label{lemma:number-of-jumps}
    In the present Lemma, all quantities are deterministic.
    Let $a \in [-\infty,+\infty)$ and $b\in(-\infty, +\infty]$ such that $a\le b$. 
    Let $\vy\in\R^n$.
    Define $\vtheta$ and $\vtheta^*$ as the unique solutions of the minimization problems
    \begin{align}
        \vtheta^* \in
        \argmin_{\vv\in\increasings} \euclidnorms{\vy - \vv},
        \label{eq:problem-vtheta-star}
        \\
        \vtheta \in
        \argmin_{\vv\in\increasings(a,b)} \euclidnorms{\vy - \vv}
        \label{eq:problem-vtheta}
    \end{align}
    where
    $
    \increasings(a,b) \coloneqq \{\vv=(v_1,...,v_n)^T \in \increasings: a\le v_1, v_n \le b\}
    $.
    Then $k(\vtheta) \le k(\vtheta^*)$.
\end{lemma}
The intuition behind this Lemma is that if a constraint is not saturated for $\vtheta$, this constraint
is not saturated for $\vtheta^*$ either, so $\vtheta^*$ has at least as many jumps as $\vtheta$.
\begin{proof}[Proof of \Cref{lemma:number-of-jumps}]
    Let $T_a = \{i=1,...,n: \hat\theta^*_i \le a \}$,
    $T_b = \{i=1,...,n: \hat\theta^*_i \ge b \}$
    and $T_c = \{i=1,...,n: a< \hat\theta^*_i < b \}$.
    We will prove that the unique minimizer $\vtheta$ of the problem \eqref{eq:problem-vtheta}
    is
    \begin{equation}
        \vtheta_{T_a} = a \vc_{T_a},
        \qquad
        \vtheta_{T_c} = \vtheta^*_{T_c},
        \qquad
        \vtheta_{T_b} = b \vc_{T_b},
        \label{eq:vtheta-from-vthetastar}
    \end{equation}
    where $\vc=(1,...,1)^T\in\R^n$.
    Then it is clear that $k(\vtheta) = 1 + k(\vtheta^*_{T_c}) + 1  \le k(\vtheta^*_{T_a}) + k(\vtheta^*_{T_c}) + k(\vtheta^*_{T_b}) = k(\vtheta^*)$.

    First, by strong convexity there exists
    a unique solution to the minimization problem \eqref{eq:problem-vtheta},
    and a unique solution to the minimization problem \eqref{eq:problem-vtheta-star}.
    Second, by the characterization of the projection onto the closed convex set  $\increasings(a,b)$,
    if $\vtheta$ satisfies
    \begin{equation}
        A_\vu \coloneqq (\vu - \vtheta)^T(\vy - \vtheta)
        \le 0
    \end{equation}
    for all $\vu\in\increasings(a,b)$,
    then $\vtheta$ is the unique solution to the minimization problem \eqref{eq:problem-vtheta}.
    Let $\vtheta$ be defined by \eqref{eq:vtheta-from-vthetastar}.
    By simple algebra, for all $\vu\in\increasings(a,b)$,
    \begin{align}
        A_\vu 
        = &
        (\vu_{T_a} - a \vc_{T_a} + \vtheta^*_{T_a} - \vtheta^*_{T_a})^T(\vy_{T_a} - \vtheta^*_{T_a})
        + (\vu_{T_a} - a \vc_{T_a})^T(\vtheta^*_{T_a} - a \vc_{T_a})
        \\
        &+ (\vu_{T_b} - b \vc_{T_b} + \vtheta^*_{T_b} - \vtheta^*_{T_b})^T(\vy_{T_b} - \vtheta^*_{T_b})
        + (\vu_{T_b} - b \vc_{T_b})^T(\vtheta^*_{T_b} - b \vc_{T_b})
        \\
        &+
        (\vu_{T_c} - \vtheta^*_{T_c})^T(\vy_{T_c} - \vtheta^*_{T_c}).
    \end{align}
    If a vector $\vv$ has nonnegative entries
    and a vector $\vx$ have non-positive entries, then $\vv^T\vx\le 0$, so 
    $(\vu_{T_a} - a \vc_{T_a})^T(\vtheta^*_{T_a} - a \vc_{T_a}) \le 0$
    and
    $(\vu_{T_b} - b \vc_{T_b})^T(\vtheta^*_{T_b} - b \vc_{T_b}) \le 0$.
    Thus,
    \begin{align}
        A_\vu 
        \le &
        (\vu_{T_a} - a \vc_{T_a} + \vtheta^*_{T_a} - \vtheta^*_{T_a})^T(\vy_{T_a} - \vtheta^*_{T_a})
        \\
        &+ (\vu_{T_b} - b \vc_{T_b} + \vtheta^*_{T_b} - \vtheta^*_{T_b})^T(\vy_{T_b} - \vtheta^*_{T_b})
        \\
        &+
        (\vu_{T_c} - \vtheta^*_{T_c})^T(\vy_{T_c} - \vtheta^*_{T_c}),
    \end{align}
    and the right hand side of the previous display is equal to
    \begin{equation}
        (\vv-\vtheta^*)^T(\vy - \vtheta^*),
        \label{eq:npquantity}
    \end{equation}
    where $\vv\in\R^n$ is defined by 
    \begin{equation}
        \vv_{T_a} \coloneqq 
        \vu_{T_a} - a \vc_{T_a} + \vtheta^*_{T_a},
        \qquad
        \vv_{T_c} \coloneqq 
        \vu_{T_c},
        \qquad
        \vv_{T_b} \coloneqq
        \vu_{T_b} - b \vc_{T_b} + \vtheta^*_{T_b}.
    \end{equation}
    We have $\vv\in\increasings$ by definition of $T_a,T_c$ and $T_b$.
    The quantity \eqref{eq:npquantity} is non-positive because $\vtheta^*$ is the projection of $\vy$ onto the convex set $\increasings$.
    Thus we have established that $A_\vu \le 0$
    for all $\vu\in\increasings(a,b)$, 
    so that the unique solution of the minimization problem \eqref{eq:problem-vtheta} is given by the expression \eqref{eq:vtheta-from-vthetastar}.
\end{proof}

\section{Proofs for convex sequences}
\label{appendix:convex}

\begin{proof}[Proof of \Cref{thm:confidence-conv}]
    For any set $T$ of the form \eqref{eq:form-T},
    using the concentration property \eqref{eq:concentration-squarednorm}
    of the random variable \eqref{eq:almost-sure-delta}
    with $K = \mathcal S^{\cap}_{|T|}$,
    we have with probability greater than $1-\alpha$,
    \begin{equation}
        \euclidnorms{\Pi_{\mathcal S^{\cap}_{|T|}}(\vxi_T)}
        \le 2 \delta\left( \mathcal S^{\cap}_{|T|}\right) + 10\log(1/\alpha)
        \le 20 \log(en) + 10\log(1/\alpha),
    \end{equation}
    where we used \eqref{eq:delta-convex} for the last inequality.
    There are less that $n^2$ sets $T\subset \{1,...,n\}$ of the form \eqref{eq:form-T}.
    Using the union bound for all sets $T$ of the form \eqref{eq:form-T},
    we have 
    $\mathbb P ( \Omega(\alpha) ) \ge 1 - \alpha$ where
    \begin{equation}
        \Omega(\alpha) \coloneqq
        \left\{
            \sup_{T\in\{T_{s,e}, 1\le s\le e\le n\}}
            \euclidnorms{\Pi_{\mathcal S^{\cap}_{|T|}}(\vxi_T)}
            \le \sigma^2 \left(20 \log(en) + 10\log\left(\frac{n^2}{\alpha} \right) \right) 
        \right\}.
    \end{equation}

    Let $\hmu=\ls(\convexs)$ for notational simplicity.
    Then
    \eqref{eq:strong-convexity-confidence} with $\vu$ replaced
    by $\vmu$ can be rewritten as
    \begin{equation}
        \euclidnorms{\hmu - \vmu}
        \le
        2 \vxi^T(\hmu-\vmu)
        - \euclidnorms{\hmu-\vmu}.
    \end{equation}
    By definition of $q(\cdot)$,
    there exists a partition $(\hat T_1,...,\hat T_{\hat q})$ of $\{1,...,n\}$
    such that
    $\ls(\convexs)$ is affine on each $\hat T_j$, $j=1,...,\hat q$.
    Furthermore, each $\hat T_j$ has the form \eqref{eq:form-T} because $\ls(\convexs)\in\convexs$.
    We have
    \begin{align}
        2 \vxi^T(\hmu-\vmu)
        - \euclidnorms{\hmu-\vmu}
        & =
        \sum_{j=1}^{\hat q }
        2 \vxi_{\hat T_j}^T(\hmu-\vmu)_{\hat T_j}
        - \euclidnorms{(\hmu-\vmu)_{T_j}}, \\
        & \le 
        \sum_{j=1}^{\hat q }
        \left( \frac{ \vxi_{\hat T_j}^T(\hmu-\vmu)_{\hat T_j} }{\euclidnorm{(\hmu-\vmu)_{\hat T_j}}} \right)^2,
    \end{align}
    where we have used 
    $2ab - a^2 \le b^2$.
    By definition of $(\hat T_1,...,\hat T_{\hat q})$,
    $\hmu$ is affine on $\hat T_j$ for each $j=1,...,\hat q$,
    thus the vector 
    $
    (\hmu-\vmu)_{\hat T_j}\in \mathcal S^{\cap}_{|\hat T_j|}
    $ is a concave sequence.
    By taking the supremum, we obtain
    \begin{equation}
        \euclidnorms{\hmu - \vmu}
        \le
        \sum_{j=1}^{\hat q}
        \sup_{\vv\in\mathcal S^{\cap}_{|\hat T_j|} : \euclidnorms{\vv} \le 1 }
        (\vxi_{\hat T_j}^T \vv)^2
        =
        \sum_{j=1}^{\hat q}
        \euclidnorms{\Pi_{\mathcal S^{\cap}_{|\hat T_j|}}(\vxi_{|\hat T_j|})}
        ,
    \end{equation}
    where we used \eqref{eq:almost-sure-delta} for the last equality.
    On the event $\Omega(\alpha)$ and by definition of $\radiusConv$,
    \begin{equation}
        \euclidnorms{\hmu - \vmu}
        \le
        \sum_{j=1}^{\hat q}
        \euclidnorms{\Pi_{\mathcal S^{\cap}_{|\hat T_j|}}(\vxi_{|\hat T_j|})}
        \le 
        \sigma^2 \hat q \left(
            20 \log(en) +  10 \log(n^2/\alpha) 
        \right)
        = n \radiusConv.
    \end{equation}
\end{proof}

\bibliographystyle{plainnat}
\bibliography{../../bibliography/db}

\end{document}